\documentclass[12pt]{amsart}
\usepackage[latin1]{inputenc}
\usepackage{indentfirst}
\usepackage{amsfonts}
\usepackage{graphicx}
\usepackage{amsmath}
\usepackage{amssymb}
\usepackage{latexsym}
\usepackage{amscd}
\usepackage[all]{xy}
\usepackage{color}
\usepackage{tikz}
\usepackage{breqn}
\usetikzlibrary{shapes.geometric}
\usepackage{algorithm}
\usepackage{algorithmic}
\usepackage{enumerate}

\def\NN{\mathbb{N}}
\def\ZZ{\mathbb{Z}} 
\def\negh{\mathbf h}

\def\negs{\mathbf s}
\def\nega{\mathbf a}

\def\negx{\mathbf x}
\def\negy{\mathbf y}
\def\negz{\mathbf z}
\def\negw{\mathbf w}
\def\nege{\mathbf e}
\def\negc{\mathbf c}

\def\negalpha{\text{\boldmath$\alpha$}}

\def\neg1{\text{\boldmath$1$}}
\def\negbeta{\text{\boldmath$\beta$}}
\def\neggamma{\text{\boldmath$\gamma$}}

\def\neggamma{\text{\boldmath$\gamma$}}

\def\neg1{\text{\boldmath$1$}}

\DeclareMathOperator{\lub}{lub}

\newtheorem{theorem}{Theorem}[section]
\newtheorem{definition}[theorem]{Definition}
\newtheorem{lemma}[theorem]{Lemma}
\newtheorem{corollary}[theorem]{Corollary}
\newtheorem{proposition}[theorem]{Proposition}
\newtheorem{remark}[theorem]{Remark}
\newtheorem{example}[theorem]{Example}

\title[The corner element of generalized numerical semigroups]{The corner element of \\ generalized numerical semigroups}

\author[M. Bernardini]{Matheus Bernardini}
\address{}
\email{matheusbernardini@unb.br}

\author[W. Tenório]{Wanderson Tenório}
\address{}
\email{wanderson\_tenorio@ufg.br}

\author[G. Tizziotti]{Guilherme Tizziotti}
\address{}
\email{guilhermect@ufu.br}

\thanks{{\em 2020 Math. Subj. Class.}: Primary 20M14; Secondary 05A15, 05A19}

\thanks{{\em Keywords}: generalized numerical semigroup, genus, corner}

\begin{document}

\maketitle

\begin{abstract}
	In this paper we introduce the concept of corner element of a generalized numerical semigroup, which extends in a sense the idea of conductor of a numerical semigroup to generalized numerical semigroups in higher dimensions. We present properties of this new notion and its relations with existing invariants in the literature, and provide an algorithm to compute all the generalized numerical semigroups with fixed corner. Besides that, we provide lower and upper bounds on the number of generalized numerical semigroups having a fixed corner element.
\end{abstract}

\section{Introduction}

Let $\mathbb{N}$ be the set of the positive integers and $\mathbb{N}_0 = \mathbb{N} \cup \{0\}$. A \textit{generalized numerical semigroup} (GNS) is a submonoid $S \subseteq \mathbb{N}_0^{d}$, where $d$ is a positive integer, such that its complement $\operatorname{H}(S)=\mathbb{N}_0^{d} \setminus S$ is finite. The elements of $\operatorname{H}(S)$ are called the \textit{gaps} (or \textit{holes}) of $S$ and its cardinality $\operatorname{g}(S)=|\operatorname{H}(S)|$ is the so-called \textit{genus} of $S$. Generalized numerical semigroups arise as a natural generalization to higher dimensions of the notion of numerical semigroup (case $d=1$), which is an active topic of research with many challenging open problems. For a detailed overview and compilation of the several ways of development on numerical semigroups, we refer the reader to \cite{GS-R, K}.

Generalized numerical semigroups were introduced by Failla, Peterson, and Utano in \cite{failla}, where they computed the number of GNSs $S \subseteq \mathbb{N}_0^{d}$ of genus $g$ for small values of $g$ and $d$ and provided certain asymptotic bounds for large values of $g$ and $d$. Since their work, several papers on GNSs, as well as on a wider class of submonoids in $\mathbb{N}_0^d$, have appeared in the literature proposing to formulate definitions, properties, results, and open problems of numerical semigroups to the general higher dimensional setting. For instance, in \cite{irreducible}, Cisto, Failla, Peterson, and Utano investigated the property of irreducibility in GNSs, introducing also the notion of Frobenius GNSs and allowable gaps. A new family of Frobenius GNSs, extending the irreducible ones, was proposed by Cisto and Ten\'orio in \cite{CT} with the study of the property of almost-symmetry for GNSs. Singhal and Lin \cite{SL} characterized the allowable gaps in GNSs and provided estimates on the number of Frobenius GNSs with a given Frobenius number. In \cite{CDGS}, Cisto, Delgado, and Garc\'ia-S\'anchez provided algorithms to perform calculations on GNS and to compute the set of all GNS with a prescribed genus. 
Pseudo-Frobenius elements of a special class of submonoids in $\mathbb{N}_0^d$ that includes the GNSs were studied by Garc\'ia-Garc\'ia, Ojeda, Rosales, and Vigneron-Tenorio in \cite{GG}. Generalizations of the Wilf's conjecture were proposed in \cite{cisto, Wilf}. An extension of proportionally modular numerical semigroups to higher dimensions is investigated in \cite{modular, modular0}.

In this work, we introduce the concept of \textit{corner} of a GNS (see Definition \ref{corner}), which somehow generalizes the notion of conductor of a numerical semigroup. We explore the properties of this new concept and its relationships with the genus and other invariants in the literature on GNS, motivated by well known relations in numerical semigroups. Besides that, using the notion of tree of GNS, we present an algorithm to compute all the GNSs with fixed corner and we provide lower and upper bounds on the number of GNSs with a fixed corner.

This paper is organized as follows. In Section 2 we present some useful definitions and notations for the rest of the paper. The concept of corner of a GNS is introduced in Section 3, where we also present properties of this concept. The relation between the genus and the corner of a GNS is studied in Section 4. In Section 5 we give an algorithm to compute all the GNSs with fixed corner. We complete this work in Section 6 by providing lower and upper bounds on the number of GNSs having fixed corner.

\section{Preliminaries and notations}
 
Throughout this paper, we use the following notations. For integers $a$ and $b$, we denote $[a] := \{x \in \ZZ: 1 \leq x \leq a\}$ and $[a,b] := \{x \in \ZZ: a \leq x \leq b\}$. For a real number $x$, $\lceil x \rceil$ stands for the smallest integer greater than or equal to $x$ and $\lfloor x\rfloor$ stands for the biggest integer smaller than or equal to $x$.

For an element $\negalpha \in \mathbb{N}_0^{d}$, the coordinates of $\negalpha$ will be denoted by $\negalpha = (\alpha_1,\ldots,\alpha_d)$ and the product of the coordinates of $\negalpha$ by the symbol $|\negalpha|$. The all zero $d$-tuple $(0, \ldots, 0)$ will be denoted simply by $\mathbf{0}$. The natural partial order $\leq$ in $\mathbb{N}_0^d$ is defined as follows: for $\negalpha, \negbeta  \in \mathbb{N}_0^{d}$, we have
$$\negalpha \leq \negbeta \ \mbox{if and only if} \ \alpha_i \leq \beta_i \  \mbox{for all} \ i \in [d].$$
For $\negalpha \in \mathbb{N}_0^d$, we consider the set $\operatorname{C}(\negalpha) := \{ \nega \in \mathbb{N}_0^{d}  \mbox{ : } \nega \leq \negalpha \}.$
Given a finite nonempty set $\mathcal{B}\subseteq \mathbb{N}_0^d$, the \emph{least upper bound} ($\lub$) of $\mathcal{B}$ is the element of $\mathbb{N}_0^d$ defined by $$\mbox{lub}(\mathcal{B}):=(\max\{\beta_1\mbox{ : } \negbeta \in \mathcal{B} \},\ldots,\max\{\beta_d\mbox{ : } \negbeta \in \mathcal{B} \}).$$

A monomial order $\prec$ is a total order in $\mathbb{N}_0^d$ that satisfies the following conditions:
\begin{itemize}
	\item for $\negalpha, \negbeta \in \mathbb{N}_0^d$, if $\negalpha \prec \negbeta$, then $\negalpha + \neggamma \prec \negbeta + \neggamma$ for all $\neggamma \in \mathbb{N}_0^d$; and
	\item  for $\negalpha \in \mathbb{N}_0^d$, if $\negalpha\neq \mathbf{0}$, we have $\mathbf{0}\prec \negalpha$.
\end{itemize}
Monomial orders extend the natural partial order $\leq$ in $\mathbb{N}_0^d$ (see \cite[Proposition~4.4]{irreducible}).

Given $S\subseteq \mathbb{N}_0^d$ a GNS, we consider the partial order $\leq_S$ in $\mathbb{N}_0^d$ defined by:
$$\negalpha \leq_S \negbeta \ \mbox{if and only if} \ \negbeta-\negalpha \in S,$$
where $\negbeta-\negalpha$ stands for the usual difference. Writing $S^*=S\setminus \{\mathbf{0}\}$, the set of pseudo-Frobenius elements of $S$ is defined as $$\operatorname{PF}(S) := \{\negx \in \operatorname{H}(S) \ : \ \negx + S^* \subset S\}.$$
Its elements are exactly the maximal elements of $\operatorname{H}(S)$ with respect to the partial order $\leq_S$ (see \cite[Proposition 1.3]{irreducible}). The set of special gaps of $S$ is $$\operatorname{SG}(S) := \{\negx \in \operatorname{PF}(S) \ : \ 2\negx \in S\}.$$
When there is a unique maximal element in $\operatorname{H}(S)$ with respect to the natural partial order $\leq$ of $\mathbb{N}_0^d$, $S$ is said to be a Frobenius GNS. Otherwise, it is said to be a non-Frobenius GNS.

\section{The corner of a GNS}

In this section, we define the corner of a GNS, which plays an important role in this paper.

\begin{definition} \label{corner}
Let $S \subseteq \mathbb{N}_0^d$ be a GNS. An element $\negc = (c_1, \ldots, c_d) \in S$ is called a corner of $S$ if the following conditions are  are satisfied:

\begin{itemize}
  \item[(1)] for all $i \in [d]$ and $\alpha \geq c_i \in \mathbb{N}_0$ we have that \newline $(\beta_1, \ldots, \beta_{i-1}, \alpha , \beta_{i+1}, \ldots , \beta_d) \in S$ for any $\beta_1, \ldots, \beta_{i-1}, \beta_{i+1}, \ldots , \beta_d \in \mathbb{N}_0$;
  \item[(2)] for all $i \in [d]$, there are $\gamma_1, \ldots, \gamma_{i-1}, \gamma_{i+1}, \ldots , \gamma_d \in \mathbb{N}_0$ such that \linebreak $(\gamma_1, \ldots, \gamma_{i-1}, c_{i} - 1, \gamma_{i+1}, \ldots , \gamma_d) \notin S$.
\end{itemize}

\end{definition}

Notice that every GNS $S$ has a corner element since $\operatorname{H}(S)$ is finite. In particular, $\mathbf{0}$ is a corner of $\mathbb{N}_0^d$.

\begin{proposition} \label{unicidade corner}
Let $S \subseteq \mathbb{N}_0^d$ be a GNS. Then the corner of $S$ is unique.
\end{proposition}

\begin{proof}
Let $\negc = (c_1, \ldots , c_d)$ and $\negc'= (c_{1}' , \ldots , c_{d} ')$ be two corners of $S$ and suppose that $\negc \neq \negc'$. Hence, there is $i \in [d]$ such that $c_i \neq c_i'$ and we can assume, without loss of generality, that $c_i - 1 \geq c_i'$. Item (1) of the Definition \ref{corner} ensures that $(\beta_1, \ldots, \beta_{i-1}, c_i - 1  , \beta_{i+1}, \ldots , \beta_d) \in S$ for any $\beta_1, \ldots, \beta_{i-1}, \beta_{i+1}, \ldots , \beta_d \in \mathbb{N}_0$, because $\negc'$ is a corner of $S$. On the other hand, item (2) of the definition guarantees that there are $\gamma_1, \ldots, \gamma_{i-1}, \gamma_{i+1}, \ldots , \gamma_d \in \mathbb{N}_0$ such that $(\gamma_1, \ldots, \gamma_{i-1}, c_{i} - 1, \gamma_{i+1}, \ldots , \gamma_d) \notin S$, which leads to a contradiction.
\end{proof}

\begin{proposition} \label{remark corner 2}
Let $S \subseteq \mathbb{N}_0^d$ be a GNS with positive genus and corner $\negc = (c_1, \ldots, c_d)$. Then the following properties hold:
\begin{itemize}
  \item[i)] $c_i \neq 0$ for all $i \in [d]$;
  \item[ii)] there is $i \in [d]$ such that $c_i>1$.
\end{itemize}
\end{proposition}

\begin{proof}
Suppose that $c_i = 0$ for some $i \in [d]$. From item (1) of Definition \ref{corner}, if $\beta_i\geq0$, then $(\beta_1, \ldots, \beta_{i-1},\beta_i,\beta_{i+1},\ldots, \beta_d) \in S$ for any $\beta_1, \ldots, \beta_{i-1}, \beta_{i+1}, \ldots , \beta_d \in \mathbb{N}_0$. Hence, $S= \mathbb{N}_0^d$ which is a contradiction. Now, suppose that $c_i = 1$ for all $i \in [d]$. Item (1) of the Definition \ref{corner} ensures that $\nege_i \in S$ for all $i\in [d]$ and, again, we conclude that $S = \mathbb{N}_0^d$, which is a contradiction.
\end{proof}

We recall that the conductor $c$ of a numerical semigroup $S$ is an element of $S$ such that $c + n \in S$, for all $n \in \NN_0$ and $c-1 \notin S$. In this way, the corner generalizes the concept of the conductor of a numerical semigroup. Indeed, for $d= 1$, the conductor and the corner are the same.

\begin{remark} \label{gaps}
Let $S \subseteq \mathbb{N}_0^d$ be a GNS with genus $g > 0$ and $\negc$ be the corner of $S$. By definition, we can conclude that $\operatorname{H}(S) \subseteq \operatorname{C}(\negc - \neg1)$, where $\neg1$ stands for the $d$-tuple $(1,\ldots, 1)$.  Moreover, the corner is the minimum element of the GNS (with respect to the partial order $\leq$) with this property, i.e., $\negc=\min_\leq \{\negx \in \NN_0^d \ : \  \operatorname{C}(\negx - \neg1)\supseteq\operatorname{H}(S)\}$.
\end{remark}

Next result relates the corner of a GNS with its set of gaps.

\begin{theorem}
Let $S \subseteq \mathbb{N}_0^{d}$ be a GNS with genus $g > 0$ and corner $\negc$. Then, 
$$\negc = \lub(\operatorname{H}(S))+\neg1.$$
In particular, $S$ is a Frobenius GNS if and only if $\negc-\mathbf{1}\in \operatorname{H}(S)$.
\label{lub}
\end{theorem}

\begin{proof}
Let $\operatorname{H}(S)=\{\negalpha_1, \ldots , \negalpha_g\}$ be the set of gaps of $S$, where $\negalpha_i =(\alpha_{1}^{(i)}, \ldots , \alpha_{d}^{(i)})$ for each $i \in [g]$. For each $j \in [d]$, define $c_j= 1 + \max \{\alpha_{j}^{(i)} \mbox{ : } i \in [g]\}$ and $\negc = (c_1,\ldots,c_d)$. The definition of $\negc$ ensures that it lies on $S$. Now we prove that $\negc$ is the corner of $S$. If $\alpha \geq c_j$, then the definition of $c_j$ guarantees that $(\beta_1, \ldots, \beta_{i-1}, \alpha , \beta_{i+1}, \ldots , \beta_d) \in S$ for all $\beta_1, \ldots, \beta_{j-1}, \beta_{j+1}, \ldots , \beta_d \in \mathbb{N}_0$ and the condition (1) in Definiton \ref{corner} is verified. Since $c_j - 1 = \max \{\alpha_{j}^{(i)} \mbox{ : } i \in [g]\}$, there is $k \in [1,g]$ such that the $j$-th coordinate of $\negalpha_{k}$ is $c_j - 1$. Hence, the condition (2) in Definiton \ref{corner} is satisfied. Therefore, $\negc$ is the corner of $S$.
\end{proof}

In particular, one can also relate the corner of a GNS with its pseudo-Frobenius elements.

\begin{corollary}	\label{pseudo}
	Let $S\subseteq \mathbb{N}_0^d$ be a GNS with genus $g > 0$ and corner $\negc$. Then
	$$\negc=\operatorname{lub}(\operatorname{PF}(S))+\mathbf{1}.$$

\end{corollary}
\begin{proof} It suffices to prove that $\operatorname{lub}(\operatorname{PF}(S))=\operatorname{lub}(\operatorname{H}(S))$. As $\operatorname{PF}(S)\subseteq \operatorname{H}(S)$, we have $\operatorname{lub}(\operatorname{PF}(S))\leq \operatorname{lub}(\operatorname{H}(S))$. On the other hand, since $\operatorname{PF}(S)$ are the maximal elements in $\operatorname{H}(S)$ with respect to $\leq_S$, for any $\negh\in \operatorname{H}(S)$ there exists $\negh'\in \operatorname{PF}(S)$ such that $\negh\leq_S \negh'$. In particular, for any $\negh\in \operatorname{H}(S)$ there exists $\negh'\in \operatorname{PF}(S)$ satisfying $\negh\leq \negh'$. Hence, $\operatorname{lub}(\operatorname{H}(S))\leq\operatorname{lub}(\operatorname{PF}(S))$.
\end{proof}

\section{The relation between the genus and the corner of a GNS}

Motivated by the relation between the genus $g$ and the conductor $c$ of a numerical semigroup by the following formula $g+1 \leq c \leq 2g$ (see \cite{GS-R}), we investigate relations between the genus of a GNS and the coordinates of its corner.

\begin{proposition} \label{cotainfgc}
Let $S \subset \NN_0^d$ be a GNS with genus $g > 0$ and corner $\negc = (c_1, \ldots, c_d)$. Then 
$$
g + 1 \leq \prod_{i=1}^{d} c_i .
$$
\end{proposition}

\begin{proof}
Using Remark \ref{gaps}, we conclude that $\operatorname{H}(S) \subseteq \operatorname{C}(\negc - \neg1)$. Since $(0,0,\ldots,0) \in S$, then $|\operatorname{H}(S)| \leq \left(\prod_{i=1}^d c_i \right) - 1$ and the result follows.
\end{proof}

In \cite{cisto}, the authors introduced the concept of GNS ordinary semigroup as follows.

\begin{definition} \label{GNS ordinary}
A GNS $S \subset \mathbb{N}_0^d$ is called ordinary if there is some $\negs \in \mathbb{N}_0^d$ such that $S = \{ \mathbf{0} \} \cup (\mathbb{N}_0^d \setminus \operatorname{C}(\negs))$.
\end{definition}

Next, we show that those GNS are the unique that reach the bound presented in Proposition \ref{cotainfgc}.

\begin{lemma} \label{lemma corner ordinary}
Let $S = \{ \mathbf{0} \} \cup (\mathbb{N}_0^d \setminus \operatorname{C}(\negs)) \subset \mathbb{N}_0^d$ be an ordinary GNS. Then the corner of $S$ is $\negc = \negs + \neg1$.
\end{lemma}

\begin{proof}
Let $\negs = (s_1, \ldots , s_d)$. If $i \in [d]$ and $\alpha \geq s_i + 1$ is an integer, then \linebreak $(\beta_1, \ldots, \beta_{i-1}, \alpha , \beta_{i+1}, \ldots , \beta_d) \notin \operatorname{C}(\negs)$, for any $\beta_1, \ldots, \beta_{i-1}, \beta_{i+1}, \ldots , \beta_d \in \mathbb{N}_0$, i.e., it belongs to $S$. Also, $\negs = (s_1, \ldots, s_{i-1}, (s_{i}+1) - 1, s_{i+1}, \ldots , s_d) \in \operatorname{C}(\negs)$, for all $i \in [d]$, hence it is not in $S$. Therefore, $\negc = (s_1 +1, \ldots , s_d + 1)$ is the corner of $S$. 
\end{proof}

The ordinary GNS with corner $\negc$ will be denoted by $\mathcal{O}(\negc)$.

\begin{proposition} \label{prop produto corner ordinary}
Let $S  \subset \mathbb{N}_0^d$ be a GNS with corner $\negc$ and genus $g > 0$. Then the following statements are equivalent: 
\begin{enumerate}[\rm (i)]
	\item $S = \mathcal{O}(\negc)$;
	\item $\prod_{i=1}^{d} c_i = g+1$.
\end{enumerate}
\end{proposition}

\begin{proof}
Let $\negc = (c_1, \ldots, c_d)$. \\
$(i) \Rightarrow (ii)$. Suppose that $S$ is an ordinary GNS with corner $\negc$ and genus $g$. So, there is some $\negs = (s_1, \ldots, s_d) \in \mathbb{N}_0^d$ such that $S = \{ \mathbf{0} \} \cup (\mathbb{N}_0^d \setminus \operatorname{C}(\negs))$. Thus, the set of gaps of $S$ is $\operatorname{H}(S) = \{\negalpha \in \mathbb{N}_0^d \mbox{ : } \mathbf{0} \neq \negalpha \leq \negs \}$, and it follows that $g=|\operatorname{H}(S)| = \prod_{i=1}^{r} (s_i + 1) - 1$. Lemma \ref{lemma corner ordinary} ensures that $s_i + 1 = c_i$.\\
$(ii) \Rightarrow (i)$. Remark \ref{gaps} guarantees that $\operatorname{H}(S) \subseteq \operatorname{C}(\negc - \neg1)$. Hence, $g = |\operatorname{H}(S)| \leq |\{ \negalpha \mbox{ ; } \mathbf{0} \neq \negalpha \leq \negc - \neg1 \}| = \prod_{i=1}^{d} c_i - 1$. By condition $(ii)$, we conclude that $\operatorname{H}(S) = \operatorname{C}(\negc - \neg1) \setminus \{\mathbf{0}\}$ and thus $S = \mathcal{O}(\negc)$.
\end{proof}

Next, we investigate a lower bound for the genus of a GNS with respect to the coordinates of its corner.

\begin{theorem}\label{existenceGNSaxes}
	Let $S\subset \mathbb{N}_0^d$ be a GNS with corner $\negc=(c_1,\ldots, c_d)$, where $c_i\geq 2$ for $i \in [d]$. There exists a GNS $S'\subset \mathbb{N}_0^d$ with corner $\negc$ such that $\operatorname{H}(S')$ is contained in the axes of $\mathbb{N}_0^d$ and $\operatorname{g}(S')\leq \operatorname{g}(S)$.
\end{theorem}
\begin{proof}  If $d = 1$, then $S' = S$. If $d > 1$, let $H_{0}:=\{\negh \in \operatorname{H}(S) \ : \ \negh \mbox{ is in the axes of } \mathbb{N}_0^d\}$ and $H_1 := \{ (h_1,\ldots , h_d) \in \operatorname{H}(S) \setminus H_0 \ : \ h_j\nege_j\in \operatorname{H}(S) \mbox{ for all } j \in [d] \}$. For \linebreak $\negh = (h_1,\ldots,h_d)\in \operatorname{H}(S)\setminus (H_{0} \cup H_1)$, define $\negh' := h_{j_0}\nege_{j_0}$, where $j_0=\min\{j\in [d] \ : \ h_j\nege_j\in S\setminus\{\mathbf{0}\}\}$. Now, taking into account the process of associating $\negh$ to $\negh'$ as above, consider the set
$$
\mathcal{H} = H_0 \cup \{ \negh' \ : \ \negh \in \operatorname{H}(S)\setminus (H_{0} \cup H_1)\}.
$$

	Note that $\mathcal{H}$ is contained in the axes of $\mathbb{N}_0^d$. Let us show that $S'=\mathbb{N}_0^d\setminus \mathcal{H}$ is a GNS in $\mathbb{N}_0^d$ with corner $\negc$. For this purpose, we will prove that if $x\nege_i\in \mathcal{H}$ with $x\nege_i=(y+z)\nege_i$ and $y\nege_i\in S'$, then  $z\nege_i\notin S'$. Now, observe that $y\nege_i\in S'$ implies that $y\nege_i\in S$, because otherwise we would have $y\nege_i\notin S'$. In the case that $z\nege_i\notin S$, we have that $z\nege_i\notin S'$ since $\mathcal{H}$ contains the gaps of $S$ in the axes of $\mathbb{N}_0^d$. On the other hand, if $z\nege_i\in S$, as $x\nege_i=(y+z)\nege_i$ and $y\nege_i\in S$, we obtain $x\nege_i\in S$. Hence, since $x\nege_i\in S$ and $x\nege_i\in \mathcal{H}$, it follows from the construction of $\mathcal{H}$ that $x\nege_i=\negh'$ for some $\negh\in \operatorname{H}(S)$ outside the axes of $\mathbb{N}_0^d$. Now, let $\negy$ be such that $\negh=y\nege_i+\negy$. By the definition, we conclude that the $i$-th coordinate of $\negy$ is $z$. Since $\negh \notin S$ and $y\nege_i \in S$, we must have $\negy\notin S$. Furthermore, because of $\negh$ is outside the axes of $\mathbb{N}_0^d$, so is $\negy$. Since the first $i-1$ coordinates of $\negh$ and $\negy$ are the same, $\negh'$ lies in the axis $Ox_i$ and $z\nege_i \in S$, we conclude that $\negy' = z\nege_i \in \mathcal{H}$. Therefore, $S'$ is a GNS.

	In order to prove that $S'$ has corner $\negc$, let us show that $\negh^i=(c_i-1)\nege_i\in \mathcal{H}$ for any $i \in [d]$. First, notice that $\negh^1\in \mathcal{H}$ because there exists $\negh\in \operatorname{H}(S)$ such that $h_1=c_1-1$, and thus we have either $\negh^1\notin S$ or $\negh^1=\negh'$. In both cases, we obtain $\negh^1\in \mathcal{H}$. Now, let us consider $\negh^i$ for $i \in [2,d]$. If $\negh^i\notin S$, then $\negh^i\in H_0 \subset \mathcal{H}$  and we are done. If $\negh^i\in S$, there exists $\negw = (w_1, \ldots, w_r) \in \operatorname{H}(S)$ such that $w_i=c_i-1$. If $w_j\nege_j\notin S$ for all $j \in [i-1]$, as $\negh^i=w_i\nege_i\in S$, then $\negh^i=\negw'$ and thus $\negh^i\in \mathcal{H}$. Now, suppose that $J = \{j \ : \ w_j\nege_j \in S \mbox{ and } j < i\}$ is a nonempty set and let $\negx=\negw-\sum_{j \in J}w_{j}\nege_{j}\notin S$. There are two possibilities: (1) if $\negx$ is in the axis of $\NN_0^d$, then $\negx \in H_0$ and thus $\negx = \negh^i$; (2) if $\negx$ is not in the axis of $\NN_0^d$: since the $i$-th coordinate of $\negx$ is $w_i=c_i-1\geq 1$ and $\negx \in \operatorname{H}(S)$, we have that $\negh^i=\negx'$ because $x_i\nege_i=\negh^i\in S$. Thus, we can conclude that $S'$ has corner $\negc$.
	
	Since every gap of $S'$ comes from at most one gap of $S$ by the construction of $\mathcal{H}$, we have the inequality $\operatorname{g}(S')\leq \operatorname{g}(S)$.
\end{proof}

\begin{example}\label{exemplocotagc}
Consider the GNS $S$ in $\mathbb{N}_0^2$ with $\operatorname{H}(S) = \{(1,0), (1,1), (3,0)\}$, which has corner $(4,2)$. We shall construct the set $\mathcal{H}$ following the proof of Theorem \ref{existenceGNSaxes}. In this case, $H_0 = \{(1,0), (3,0)\}, H_1 = \emptyset$ and $\mathcal{H} = \{(1,0), (3,0)\} \cup \{(1,1)'\}$. Theorem \ref{existenceGNSaxes} ensures that $S' = \NN_0^2 \setminus \mathcal{H}$ is a GNS with corner $(4,2)$ and now we explicit it by the set of gaps.  By definition, $(1,1)' = (0,1)$, since $(1,0) \notin S$ and $(0,1) \in S$. Hence, $\mathcal{H} = \{(1,0), (3,0), (0,1)\}$ is the set of gaps of $S'$, which is a GNS with all the gaps in the axis. Moreover, $\operatorname{g}(S) = 3$ and $\operatorname{g}(S') = 3$.
\end{example}

\begin{example}
Consider the GNS $S$ in $\mathbb{N}_0^4$ with 
\begin{center}
$\operatorname{H}(S)=\{(0, 0, 0, 1), (0, 0, 1, 0), (0, 0, 1, 1), (0, 0, 2, 0), (0, 0, 3, 0), (0, 0, 3, 1), (0, 1, 1, 1),$
$(1, 0, 0, 0), (1, 0, 0, 1), (1, 0, 0, 3), (1, 0, 1, 0), (1, 0, 2, 0), (1, 0, 2, 2), (1, 0, 6, 0), (3, 0, 0, 0)\}.$
\end{center}

Observe that $S$ has corner $(4,2,4,4)$. 	This example also illustrates the method of obtaining $S'$ from $S$ as in Theorem \ref{existenceGNSaxes}. In this case, 
$$H_0 = \{(0, 0, 0, 1), (0, 0, 1, 0), (0, 0, 2, 0), (0, 0, 3, 0), (1, 0, 0, 0), (3, 0, 0, 0)\},$$ 
$$H_1 = \{(0, 0, 1, 1), (0, 0, 3, 1), (1, 0, 0, 1),
(1, 0, 1, 0), (1, 0, 2, 0)\}$$
and the set $\mathcal{H}$ is

\begin{center}
$\mathcal{H} =\{(0, 0, 0, 1), (0, 0, 1, 0), (0, 0, 2, 0), (0, 0, 3, 0), (1, 0, 0, 0), (3, 0, 0, 0)\} \cup$
$\{(0, 1, 1, 1)', (1, 0, 0, 3)', (1, 0, 2, 2)', (1, 0, 6, 0)'\}$.
\end{center}

Theorem \ref{existenceGNSaxes} ensures that $S' = \NN_0^2 \setminus \mathcal{H}$ is a GNS with corner $(4,2,4,4)$ and its set of gaps is

\begin{center}
$\operatorname{H}(S')=\{(0, 0, 0, 1), (0, 0, 1, 0), (0, 0, 2, 0), (0, 0, 3, 0),  (1, 0, 0, 0), (3, 0, 0, 0)$, 
$(0, 1, 0, 0), (0, 0, 0, 3), (0, 0, 0, 2), 
(0, 0, 6, 0)\}.$
\end{center}

Furthermore, $\operatorname{g}(S) = 15$ and $\operatorname{g}(S') = 10$.
\end{example}

Next, we present a lower bound for the genus of a GNS in terms of the coordinates of its corner.

\begin{proposition}\label{cotainferiorgenero}
Let $S \subset \NN_0^d$ be a GNS with corner $\negc = (c_1, \ldots, c_d)$, with $c_i \geq 2$ for all $i$. Then 
$$
\sum_{i=1}^{d} c_i \leq 2\operatorname{g}(S).
$$
\end{proposition}

\begin{proof}
We want to minimize $\operatorname{g}(S)$, for $S$ in the set of all GNS with fixed corner $\negc$. By Theorem \ref{existenceGNSaxes}, we only have to check those GNS with all the gaps in the axes of $\NN_0^d$.

Let $S$ be a GNS with all the gaps in the axes of $\NN_0^d$ and consider $S_i := \{s \in \NN_0 \ : \ s\nege_i \in S\}$. One can check that $S_i$ is a numerical semigroup with conductor $c_i$. Let $g_i$ be the genus of $S_i$. By numerical semigroups properties, we obtain $c_i \leq 2g_i$, for all $i$ and the genus of $S$ is given by $\sum g_i$. Hence, $\sum c_i \leq \sum 2g_i = 2\operatorname{g}(S)$ and we are done.
\end{proof}

Notice that both GNS given in Example \ref{exemplocotagc} are examples that reach this last bound.

\begin{remark}
The arithmetic-geometric mean inequality guarantees that if $S \subset \NN_0^d$ is a GNS with genus $g$ and corner $\negc = (c_1, \ldots, c_d)$, with $c_i \geq 2$, for all $i$, then
$$\prod_{i=1}^{d} c_i \leq \left(\frac{1}{d} \cdot \sum_{i=1}^{d} c_i\right)^d \leq \left(\frac{2g}{d}\right)^d.$$
\end{remark}

Next, we exhibit a GNS with corner $\negc$ with the least possible genus. For this propose, we deal with irreducible numerical semigroups. Recall that a numerical semigroup with genus $g$ and conductor $c$ satisfies $g \geq \left \lceil \frac{c}{2} \right \rceil$. Recall furthermore that a numerical semigroup is irreducible if, and only if, $g = \left \lceil \frac{c}{2} \right \rceil$ (cf. \cite{GS-R}). Moreover, for each $c \in \NN, c \geq 2$, there is an irreducible numerical semigroup with conductor $c$.

\begin{corollary}\label{irreducible}
	Let $\negc=(c_1,\ldots, c_d) \in \NN_0^d$, where $c_i\geq 2$ for all $i$. There exists a GNS $T$ with corner $\negc$, such that
	$$\operatorname{g}(T)= \sum_{i=1}^d \left\lceil \frac{c_i}{2}\right\rceil.$$
	Moreover, this is the least possible genus for a GNS with corner $\negc$.
\end{corollary}

\begin{proof}
For each $i \in [d]$, let $T_i$ be an irreducible numerical semigroup with conductor $c_i$. By taking $\mathcal{H} := \bigcup_{i=1}^{d} \{h \nege_i \ : \ h \notin T_i\}$, one can check that $T = \NN_0^d \setminus \mathcal{H}$ is a GNS with genus $\sum_{i=1}^d \left\lceil \frac{c_i}{2}\right\rceil$.

Now, let $S$ be a GNS with corner $\negc$. From Theorem \ref{existenceGNSaxes}, there is a GNS $S'$ with all gaps in the axes such that $\operatorname{g}(S') \leq \operatorname{g}(S)$. For each $i \in [d]$, consider the numerical semigroup $S'_i := \{s \in \NN_0 \ : \ s\nege_i \in S'\}$ with genus $g_i$. From the construction of $S'_i$, its conductor is $c_i$ (since the corner of $S'$ is $\negc$) and the genus of $S'$ is $\sum g_i$. Since $g_i \geq  \left \lceil \frac{c_i}{2} \right \rceil$ for each $i \in [d]$ and we conclude that
$$\operatorname{g}(S) \geq \operatorname{g}(S') \geq \sum_{i=1}^d \left\lceil \frac{c_i}{2}\right\rceil.$$

\end{proof}


To end this section, we explain the reason for the hypothesis $c_i \geq 2$, for every $i$ in last results. We also explain how we can obtain a relation between the sum of the coordinates of the corner and the genus of a GNS, if some of the coordinates of the corner are equal one. For instance, the GNS $S = \mathbb{N}_0^2 \setminus \{(1,0)\}$ has genus $g = 1$ and corner $(2,1)$; the sum of the coordinates of the corner is greater than twice the genus. Hence, Proposition \ref{cotainferiorgenero} does not hold in this case. 

\begin{remark}
Let $S \subset \mathbb{N}_0^d$ be a GNS with genus $g$ and corner $\negc = (c_1, \ldots, c_d)$ such that the set of indexes $\mathcal{U}(\negc) = \{j \ : \ c_j = 1\}$ is nonempty. We want to obtain an upper bound for sum of the coordinates of $\negc$ in terms of $g$, where $|\mathcal{U}(\negc)| = k$ is a positive number. By the definition of corner, we conclude that $k \leq d - 1$.  Notice that all the gaps of $S$ are of the form $\negh = (h_1, \ldots, h_d)$, where $h_i = 0$, if $i \in \mathcal{U}(\negc)$. Hence, we can look at the set $H(S)$ as a subset of $\NN_0^{d-k}$, by erasing all the coordinates that are in a $j$-th position, with $j \in J$. This new set is a GNS in $\NN_0^{d-k}$ with corner $\bar{\negc}$, which has all the coordinates greater than one. Moreover, the genus of this new GNS is the same as the genus of $S$. Hence, we can apply Proposition \ref{cotainferiorgenero}. In this case, the sum of the coordinates of the corner $\bar{\negc}$ is such that $\sum_{i \notin \mathcal{U}(\negc)} c_i \leq 2g$. By summing up $\sum_{i \in \mathcal{U}(\negc)} c_i$ in both sides and recalling that $c_i = 1$, for $i \in \mathcal{U}(\negc)$, we conclude that

$$\sum_{i = 1}^d c_i \leq 2g + k,$$

which is globally bounded by $2g + d - 1$.
\end{remark}

\section{The tree of GNS with fixed corner}

In this section, we give an algorithm to compute all the GNSs with fixed corner. Consider $\mathcal{C}(\negc)$ the family of GNSs having corner $\negc$. From Proposition \ref{remark corner 2}, we shall consider $\negc \in \mathbb{N}^d \setminus \{\mathbf{1}\}$ and one can take the ordinary GNS $\mathcal{O}(\negc)$ in $\mathcal{C}(\negc)$. Thus, we present a procedure to obtain all elements in $\mathcal{C}(\negc)$ from $\mathcal{O}(\negc)$. In special, we show that this method allow us arranging all GNSs having fixed corner into a rooted tree.\\

For $\negx\in \mathbb{N}_0^d$, recall that $\operatorname{C}(\negx)=\{\negy\in \NN_0^d \ : \ \negy\leq \negx\}$. Given a GNS $S$ and $\negx \in \NN_0^d$ such that $\operatorname{C}(\negx-\mathbf{1})\supset\operatorname{H}(S)$, define for each $i\in [d]$ the sets $$\nabla_i(S,\negx):=\{\negh\in \operatorname{H}(S)  \ : \ h_i= x_i-1 \ \mbox{and } h_j\leq x_j-1 \ \mbox{for } j\neq i\},$$
where $\negh = (h_1, \ldots, h_d)$ and $\negx = (x_1, \ldots, x_d)$.  Next, we characterize GNSs with fixed corner $\negc$ in terms of the sets $\nabla_i(S,\negc)$.

\begin{lemma}\label{cornernabla}
	Let $S\subseteq \NN_0^d$ be a GNS and let $\negc \in \NN_0^d$ such that $\operatorname{C}(\negc-\mathbf{1})\supset\operatorname{H}(S)$. Then $S$ has corner $\negc$ if and only if $\nabla_i(S,\negc)\neq \emptyset$ for all $i\in [d]$.
\end{lemma}
\begin{proof}
	If $\nabla_i(S,\negc)= \emptyset$ for some $i\in [d]$, then there is no gap of $S$ such that its $i$-th coordinate is $c_i - 1$. Hence, the $i$-th coordinate of $\operatorname{lub}(\operatorname{H}(S))$ is smaller than $c_i - 1$, contradicting Theorem \ref{lub}. On the other hand, if $\nabla_i(S,\negc)\neq \emptyset$ for all $i\in [d]$, then $\negc=\operatorname{lub}(\operatorname{H}(S))+\mathbf{1}$ is the corner of $S$ by Theorem \ref{lub}.
\end{proof}

We now consider unitary extensions of GNSs which preserve the property of having a fixed corner $\negc$.  Recall that if $S$ is a GNS and $\negx \notin S$, then $S \cup \{\negx\}$ is a GNS if and only if $\negx \in \operatorname{SG}(S)$ (see \cite{irreducible} Proposition 2.3).

\begin{proposition}\label{childcorner}
	Let $S\subseteq \NN_0^d$ be a GNS with corner $\negc$ and $\negx \in \operatorname{SG}(S)$. Then $S\cup \{\negx\}$ has corner $\negc$ if and only if $\nabla_i(S, \negc)\neq \{\negx\}$ for all $i\in [d]$.
\end{proposition}
\begin{proof}
	By Lemma~\ref{cornernabla}, it suffices to notice that for each $i\in [d]$ we have $\nabla_i(S\cup \{\negx\}, \negc)\neq \emptyset$  if and only if $\nabla_i(S, \negc)\neq \{\negx\}$.
\end{proof}

Recall that if $T$ is a GNS and $\negx \in T$, then $T\setminus \{\negx\}$ is a GNS if and only if $\negx$ is a minimal generator of $T$, that is, $\negx \in T^*\setminus (T^* +T^*)$ (see \cite{failla} Proposition 4.1). We now look for conditions on a minimal generator of a GNS so that the new GNS obtained by taking it out has the same corner as the previous one.

\begin{proposition}
	Let $T\subseteq \NN_0^d$ be a GNS with corner $\negc$ and let $\negx $ be a minimal generator of $T$. Then $T\setminus \{\negx\}$ has corner $\negc$ if and only if $\negx\leq \negc-\mathbf{1}$.
	\label{parentcorner}
\end{proposition}
\begin{proof} If $T\setminus \{\negx\}$ has corner $\negc$, as $\negx\in \operatorname{H}(T\setminus \{\negx\})$, then $\negx\leq \negc-\mathbf{1}$. Conversely, if $\negx\leq \negc-\mathbf{1}$, then $\operatorname{H}(T)\cup \{\negx\}\subset \operatorname{C}(\negc-\mathbf{1})$. Since $\operatorname{H}(T\setminus \{\negx\})=\operatorname{H}(T)\cup \{\negx\}$ and $\nabla_i(T\setminus \{\negx\}, \negc)\supseteq\nabla_i(T, \negc)$ for all $i\in [d]$, and the result follows from Lemma~\ref{cornernabla}.
\end{proof}

\begin{remark} \label{remextensions} The unitary extensions of GNSs with corner $\negc$ given in Propositions \ref{childcorner} and \ref{parentcorner} are inverse procedures to each other in the sense that, for $S$ and $T$ GNSs with corner $\negc$, we have:
	\begin{itemize}
		\item if $\negx \in \operatorname{SG}(S)$, then $\negx$ is a minimal generator of $S\cup \{\negx\}$ with $\negx\leq \negc-\mathbf{1}$;  and
		\item if $\negx$ is a minimal generator of $T$ with $\negx\leq \negc-\mathbf{1}$, then $\negx\in \operatorname{SG}(T\setminus \{\negx\})$.
	\end{itemize}
\end{remark} 

In order to obtain GNSs with a same corner by adding special gaps, motivated by Proposition~\ref{childcorner}, let us consider the following definition.

\begin{definition}
	For $S\subseteq \NN_0^d$ a GNS with corner $\negc$, define 
	$$\operatorname{D}(S):=\{\negx \in \operatorname{SG}(S) \ : \ \nabla_i(S, \negc)\neq \{\negx\} \ \mbox{for all } i\in [d]\}.$$
\end{definition} 

Next, we describe a procedure to obtain all the GNSs in $\mathcal{C}(\negc)$. The main idea is building up GNSs with corner $\negc$ from $\mathcal{O}(\negc)$ (the ordinary GNS with corner $\negc$) through Proposition~\ref{childcorner} by considering unitary extensions $S\cup \{\negx\}$ for elements $\negx\in \operatorname{D}(S)$, where $S$ is a GNS with corner $\negc$. However, this method may provide redundant GNSs in $\mathcal{C}(\negc)$, in the sense that one can occur that a same GNS to be generated more than one time in this way. To avoid this situation, we will consider a special subset of $\operatorname{D}(S)$ as follows.

\begin{definition}
	Let $S\subseteq \NN_0^d$ be a GNS with corner $\negc$ and let $\prec$ be a monomial order. Considering $\operatorname{L}(S):=\{\negx \in S \ : \ \negx\leq \negc-\mathbf{1}\}$, define
	$$\operatorname{D}_{\prec}(S):=\{\negx \in \operatorname{D}(S) \ : \ \negx \prec \negy\ \mbox{for all}\ \negy \in \operatorname{L}(S)\setminus \{\mathbf{0}\}\}.$$
	If $S$ is a nonordinary GNS, we also define the element  $$\operatorname{low}_{\prec}(S):=\min_{\prec}(\operatorname{L}(S)\setminus \{\mathbf{0}\}).$$
\end{definition} 

Notice that, except for $\mathcal{O}(\negc)$, it is always ensured the existence of a such minimal generator $\negx$ in a GNS with corner $\negc$ as in Proposition~\ref{parentcorner}. 

\begin{lemma}
	Let $T\subseteq \NN_0^d$ be a nonordinary GNS with corner $\negc$. Let $\prec$ be a monomial order and $\negx=\operatorname{low}_{\prec}(T)$. Then $T\setminus\{\negx\}$ is a GNS with corner $\negc$.
	\label{xmin}
\end{lemma}
\begin{proof}
	Since $T$ is nonordinary, the set $\operatorname{L}(T)\setminus \{\mathbf{0}\}=\{\negz \in T^* \ : \ \negz \leq  \negc-\mathbf{1}\}$ is not empty, and thus the element $\negx$ is well-defined. Notice that $\negx$ is a minimal generator of $T$ since otherwise we could write $\negx=\negx_1+\negx_2$ with $\negx_1,\negx_2\in T^*$, which implies that $\negx_1\prec\negx$, contradicting the minimality of $\negx$ with respect to $\prec$ in $\operatorname{L}(S)\setminus \{\mathbf{0}\}$ because $\negx_1\leq \negx\leq \negc-\mathbf{1}$. Hence, $T\setminus\{\negx\}$ is a GNS that has corner $\negc$ by Proposition~\ref{parentcorner}.
\end{proof}

\begin{lemma}\label{alg}
	Let $S\subseteq \NN_0^d$ be a nonordinary GNS with corner $\negc$ and let $\prec$ be a monomial order. Then there exists a chain $S_{1}\supset S_{2}\supset\cdots \supset S_{n-1} \supset S_{n}$ of GNSs with corner $\negc$ such that:
	\begin{itemize}
		\item $S_1=S$;
		\item $S_{i+1}=S_{i}\setminus \{\operatorname{low}_{\prec}(S_i)\}$ for $i \in [n-1]$; and in particular 
		\item $S_{n}=\mathcal{O}(\negc)$.
	\end{itemize}
	\label{path}
\end{lemma}
\begin{proof}
	For $S_{1}=S$, it follows from Lemma~\ref{xmin} that $S_1\setminus \{\negx_1\}$ has corner $\negc$, for $\negx_1 =\operatorname{low}_{\prec}(S_1)$. Putting $S_{2}=S_1\setminus \{\negx_1\}$, if $S_2$ is ordinary we conclude the procedure, and otherwise we consider $S_3=S_2\setminus \{\negx_2\}$ with corner $\negc$, where $\negx_2 =\operatorname{low}_{\prec}(S_2)$, by Lemma~\ref{xmin}. Repeating this argument for each $i\geq 2$, we obtain a GNS  $S_{i+1}=S_{i}\setminus \{\negx_i\}$ with corner $\negc$, where $\negx_i=\operatorname{low}_{\prec}(S_i)$. The procedure stops when it reaches $S_{i}=\mathcal{O}(\negc)$ for some $i$ (and it occurs because $\operatorname{L}(S)$ is finite).
\end{proof}

\begin{remark}\label{remlow}
	It is worth to mention that, in the previous result, the element $\operatorname{low}_{\prec}(S_i)\in \operatorname{D}_\prec(S_{i+1})$ for each $i$. Indeed, as observed in Remark~\ref{remextensions},  $\operatorname{low}_{\prec}(S_i)\in \operatorname{SG}(S_{i+1})$. Furthermore, we get $\nabla_j(S_{i+1}, \negc)\neq \{\operatorname{low}_{\prec}(S_i)\}$ for all $j\in [d]$,  by using that $S_i=S_{i+1}\cup\{\operatorname{low}_{\prec}(S_i)\}$ has corner $\negc$ in Proposition~\ref{childcorner}. Hence, we have  $\operatorname{low}_{\prec}(S_i)\in \operatorname{D}(S_{i+1})$. As $\operatorname{low}_{\prec}(S_i)=\min_{\prec}(\operatorname{L}(S_{i})\setminus\{\mathbf{0}\})$, we get $\operatorname{low}_{\prec}(S_i)\prec \negy$ for all $\negy\in \operatorname{L}(S_{i+1})\setminus\{\mathbf{0}\}$ because $\operatorname{L}(S_{i})\supset\operatorname{L}(S_{i+1})$.
\end{remark}

As a consequence, gathering the conclusions of Proposition~\ref{cotainfgc}, Corollary~\ref{irreducible} and Lemma~\ref{alg}, we obtain the distribution of genera of GNSs with prescribed corner.

\begin{corollary}
		Given a pair $(g, \negc)\in \mathbb{N}\times (\NN^d_0\setminus \operatorname{C}(\mathbf{1}))$, there is a GNS with corner $\negc$ and genus $g$ if, and only if, 
		$$\textstyle \left\lceil \frac{c_1}{2}\right\rceil+\cdots+\left\lceil \frac{c_d}{2}\right\rceil\leq g\leq |\negc|-1.$$
\end{corollary}

In order to provide a procedure that gives all GNSs having fixed corner $\negc$, without repetitions of GNSs, let us consider the following definition.

\begin{definition}
	Let $\prec$ be a monomial order, $\negc\in \NN^d$ and let $\mathcal{C}(\negc)$ be the set of GNSs having corner $\negc$. Define $\mathcal{G}_{\prec}(\negc)=(\mathcal{C}(\negc), \mathcal{E}_{\prec})$ to be the graph whose set of vertices is $\mathcal{C}(\negc)$ and the set of edges is  $\mathcal{E}_{\prec}:=\{(T,S)\in \mathcal{C}(\negc)\times \mathcal{C}(\negc) \ : \ S=T\setminus \{\operatorname{low}_{\prec}(T)\}\}$. 
	If $(T,S)\in \mathcal{E}_{\prec}$, $T$ is called a \emph{child} of $S$. 
	\label{defTree}
\end{definition}

The following result shows us it is possible to arrange GNSs having fixed corner into a rooted tree.

\begin{theorem} \label{thmTree}
	Let $\prec$ be a monomial order and $\negc\in \NN^d$. Then $\mathcal{G}_{\prec}(\negc)$ is a rooted tree whose root is $\mathcal{O}(\negc)$ and the children of $S\in \mathcal{C}(\negc)$ are  $S\cup \{\negx\}$, where $\negx \in \operatorname{D}_{\prec}(S)$.
\end{theorem}

\begin{proof}
	Given $S\in \mathcal{C}(\negc)$, it follows from Lemma~\ref{path} that there exists a chain of GNSs $S_{1}\supset S_{2}\supset\cdots \supset S_{n-1} \supset S_{n}$ such that $S_1=S$,  $S_{i+1}=S_{i}\setminus \{\operatorname{low}_{\prec}(S_i)\}$ and $S_n=\mathcal{O}(\negc)$. In particular, $(S_1,S_2),(S_2,S_3),\ldots,(S_{n-1},S_n)$ is a path of edges of $\mathcal{G}_{\prec}(\negc)$ linking $S$ to $\mathcal{O}(\negc)$. If other path of edges of $\mathcal{G}_{\prec}(\negc)$ links $S$ to $\mathcal{O}(\negc)$, then for some $i\in [n]$ there are two different GNSs $T_1,T_2 \in \mathcal{C}(\negc)$ such that $(S_i,T_1),(S_i,T_2)\in \mathcal{E}_{\prec}$. By the definition of $\mathcal{G}_{\prec}(\negc)$, we have $T_1=S_i\setminus \{\operatorname{low}_{\prec}(S_i)\}=T_2$, which contradicts $T_1\neq T_2$. Hence, we conclude that $\mathcal{G}_{\prec}(\negc)$ is a rooted tree whose root is the vertex $\mathcal{O}(\negc)$. Now, if $T$ is a child of $S$, then $S=T\setminus\{\operatorname{low}_{\prec}(T)\}$, and therefore $T= S\cup \{\operatorname{low}_{\prec}(T)\}$. In particular, following the same idea of Remark~\ref{remlow}, we obtain that $\operatorname{low}_{\prec}(T)\in \operatorname{D}_{\prec}(S)$, which proves the result.
\end{proof}

Observe that different monomial orders $\prec_1$ and $\prec_2$ in $\mathbb{N}_0^d$ may provide different trees $\mathcal{G}_{\prec_1}(\negc)$ and $\mathcal{G}_{\prec_2}(\negc)$, although both sets of vertices are the same.

\begin{example}\label{figures}
	Given $\negc=(3,2)$, let us compute the rooted tree $\mathcal{G}_{\prec}(\negc)$ of GNSs having corner $\negc$ by considering the lexicographic order. Since a such GNS $S$ is entirely described by the set $\operatorname{L}(S)$, we shall use those sets in the Figure 1 to illustrate the GNSs in the rooted tree. In Figure 1, the elements of an $S\in \mathcal{C}(\negc)$ are denoted by the black points and those of $\operatorname{H}(S)$ are denoted by the red ones.
\end{example}

	\begin{figure}[h] \label{fig1}
		\begin{center}
			\begin{tikzpicture}[sibling distance=10em]
			\node {\begin{tikzpicture}[scale=.5]
				\draw [mark=*, color=red] plot (1,0);
				\draw [mark=*, color=red] plot (0,1);
				\draw [mark=*] plot (0,0);
				\draw [mark=*, color=red] plot (2,0);
				\draw [mark=*, color=red] plot (1,1);
				\draw [mark=*, color=red] plot (2,1);
				\end{tikzpicture}} [sibling distance=4.2cm]
			child { node {\begin{tikzpicture}[scale=.5]
					\draw [mark=*, color=red] plot (1,0);
					\draw [mark=*] plot (0,1);
					\draw [mark=*] plot (0,0);
					\draw [mark=*, color=red] plot (2,0);
					\draw [mark=*, color=red] plot (1,1);
					\draw [mark=*, color=red] plot (2,1);
					\end{tikzpicture}} 
			}
			child { node {\begin{tikzpicture}[scale=.5]
					\draw [mark=*, color=red] plot (1,0);
					\draw [mark=*, color=red] plot (0,1);
					\draw [mark=*] plot (0,0);
					\draw [mark=*, color=red] plot (2,0);
					\draw [mark=*] plot (1,1);
					\draw [mark=*, color=red] plot (2,1);
					\end{tikzpicture}} 
				child { node {\begin{tikzpicture}[scale=.5]
						\draw [mark=*, color=red] plot (1,0);
						\draw [mark=*] plot (0,1);
						\draw [mark=*] plot (0,0);
						\draw [mark=*, color=red] plot (2,0);
						\draw [mark=*] plot (1,1);
						\draw [mark=*, color=red] plot (2,1);
						\end{tikzpicture}} }}
			child { node {\begin{tikzpicture}[scale=.5]
					\draw [mark=*, color=red] plot (1,0);
					\draw [mark=*, color=red] plot (0,1);
					\draw [mark=*] plot (0,0);
					\draw [mark=*] plot (2,0);
					\draw [mark=*, color=red] plot (1,1);
					\draw [mark=*, color=red] plot (2,1);
					\end{tikzpicture}} [sibling distance=2.2cm]
				child { node {\begin{tikzpicture}[scale=.5]
						\draw [mark=*] plot (1,0);
						\draw [mark=*, color=red] plot (0,1);
						\draw [mark=*] plot (0,0);
						\draw [mark=*] plot (2,0);
						\draw [mark=*, color=red] plot (1,1);
						\draw [mark=*, color=red] plot (2,1);
						\end{tikzpicture}} }
				child { node {\begin{tikzpicture}[scale=.5]
						\draw [mark=*, color=red] plot (1,0);
						\draw [mark=*, color=red] plot (0,1);
						\draw [mark=*] plot (0,0);
						\draw [mark=*] plot (2,0);
						\draw [mark=*] plot (1,1);
						\draw [mark=*, color=red] plot (2,1);
						\end{tikzpicture}} }}
			child { node {\begin{tikzpicture}[scale=.5]
					\draw [mark=*, color=red] plot (1,0);
					\draw [mark=*, color=red] plot (0,1);
					\draw [mark=*] plot (0,0);
					\draw [mark=*, color=red] plot (2,0);
					\draw [mark=*, color=red] plot (1,1);
					\draw [mark=*] plot (2,1);
					\end{tikzpicture}} [sibling distance=2.2cm]
				child { node {\begin{tikzpicture}[scale=.5]
						\draw [mark=*, color=red] plot (1,0);
						\draw [mark=*] plot (0,1);
						\draw [mark=*] plot (0,0);
						\draw [mark=*, color=red] plot (2,0);
						\draw [mark=*, color=red] plot (1,1);
						\draw [mark=*] plot (2,1);
						\end{tikzpicture}} }
				child { node {\begin{tikzpicture}[scale=.5]
						\draw [mark=*, color=red] plot (1,0);
						\draw [mark=*, color=red] plot (0,1);
						\draw [mark=*] plot (0,0);
						\draw [mark=*, color=red] plot (2,0);
						\draw [mark=*] plot (1,1);
						\draw [mark=*] plot (2,1);
						\end{tikzpicture}} } };
		\end{tikzpicture}
	\end{center}
	\caption{The tree of GNSs in $\mathbb{N}_0^2$ with fixed corner $\negc=(3,2)$, with respect to the lexicographical order, of Example~\ref{figures}.}
\end{figure}
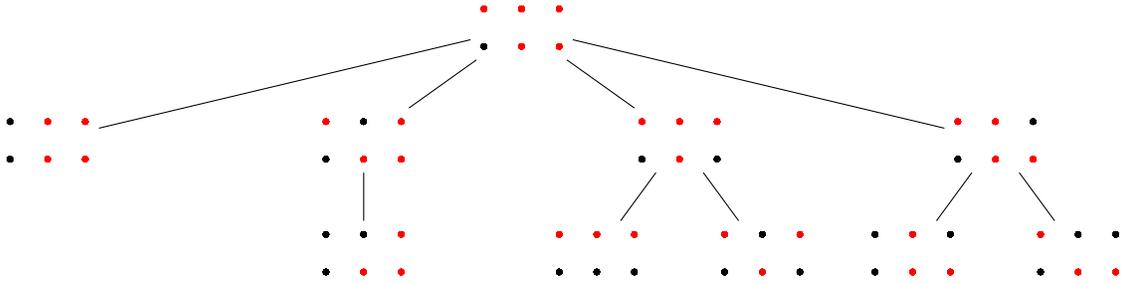

Hence, we have a procedure that computes all GNSs with corner $\negc$, without repetitions, which relies on Theorem \ref{thmTree}. It is presented in Algorithm 1 as follows.

\begin{algorithm}
	\caption{Algorithm for computing the set $\mathcal{C}(\negc)$ of all GNSs with fixed corner $\negc$}
	\begin{algorithmic} 
		\REQUIRE $\negc\in \mathbb{N}^d\setminus \{\mathbf{1}\}$ and a monomial order $\prec$ in $\mathbb{N}_0^d$.
		\ENSURE the set $\mathcal{C}(\negc)$ of all GNSs with fixed corner $\negc$.
		\STATE Compute the ordinary GNS $\mathcal{O}(\negc)$.
		\STATE Set $\mathcal{C}(\negc)=\{\mathcal{O}(\negc)\}$.
		\STATE $n=1$.	
		\STATE Set $\mathcal{I}[n]=\emptyset$.
		\FOR{ $\negx\in \operatorname{SG}(\mathcal{O}(\negc))$}
		\STATE Append $\mathcal{O}(\negc)\cup \{\negx\}$ to $\mathcal{I}[1]$.
		\STATE Append $\mathcal{O}(\negc)\cup \{\negx\}$ to $\mathcal{C}(\negc)$.
		\ENDFOR
		\WHILE{$\mathcal{I}[n]\neq \emptyset$}
		\STATE $\mathcal{I}[n+1]=\emptyset$.
		\FOR{ $S\in \mathcal{I}[n]$}
		\STATE Compute $\operatorname{D}_{\prec}(S)$.
		\FOR { $\negx\in \operatorname{D}_{\prec}(S)$}
		\STATE Append $S\cup \{\negx\}$ to $\mathcal{C}(\negc)$.
		\STATE Append $S\cup \{\negx\}$ to $\mathcal{I}[n+1]$.
		\ENDFOR
		\ENDFOR
		
		$n\leftarrow n+1$.
		\ENDWHILE
	\end{algorithmic}
\end{algorithm}

\section{Bounds on the number of GNS with fixed corner}

In this section, we provide lower and upper bounds on the number of GNSs having fixed corner. Since the notions of conductor and corner coincide for $d=1$, the well known bounds due to Backelin \cite{backelin} for the number of numerical semigroups with fixed Frobenius number give us naturally the following bounds for the number $N(c)$ of numerical semigroups with corner $c$ as
$$2^{\lfloor\frac{c-2}{2}\rfloor}\leq N(c)\leq 4\cdot 2^{\lfloor\frac{c-2}{2}\rfloor}.$$
As the lower bound above comes up from the observation that every subset $A$ of $\{n\in \mathbb{N} \ : \ \lceil\frac{c}{2}\rceil\leq n < c-1\}$ provides a numerical semigroup with conductor $c$ by considering $A\cup \mathcal{O}(c)$, where $\mathcal{O}(c)=\{0, c, c+1,...\}$ is the ordinary numerical semigroup with conductor $c$, we will employ a generalization of this idea to give a lower bound for the number of GNSs in $\mathbb{N}_0^d$, with $d\geq 2$, having fixed corner $\negc\in \mathbb{N}_0^d\setminus \operatorname{C}(\mathbf{1})$. 

\subsection{A lower bound on the number of GNSs with fixed corner}

Let $\mathcal{P}_d$ be the power set of $[d]$. For any $J \in \mathcal{P}_d$ and $\negy=(y_1,\ldots, y_d)\in \mathbb{N}_0^d\setminus \operatorname{C}(\mathbf{1})$, we define
\begin{center}
	$
	\Omega_{J}(\negy):= \big\{ \negx \in \mathbb{N}_0^d \ : \ \big\lceil \frac{y_{j}}{2} \big\rceil  \leq x_{j} \leq  y_{j} - 1 \mbox{ for } j \in J \mbox{ and } x_i < \big\lceil \frac{y_i}{2} \big\rceil \mbox{ for } i \in [d] \setminus J  \big\}
	$.
\end{center}

\begin{remark} \label{remark Omega_i}
	Observe that for any two distinct $J, J'\in \mathcal{P}_d$, the sets $\Omega_{J}(\negy)$ and $\Omega_{J'}(\negy)$ are disjoint. Furthermore,  these sets $\Omega_{J}(\negy)$ split the region $\operatorname{C}(\negy-\mathbf{1})$ of $\mathbb{N}_0^d$ into $2^d$ disjoint subsets. The Figures 3 and 4 illustrate such decomposition in $\mathbb{N}_0^2$ and $\mathbb{N}_0^3$:
	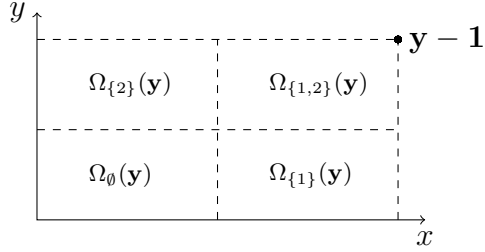
\begin{figure}[htp]
		
		\begin{tikzpicture}[scale=0.6] 
		\draw [<->] (0,4.6) node [left] {$y$} -- (0,0)
		-- (8.6,0) node [below] {$x$};

		\draw (.9,1) node [right] {{\scriptsize $\Omega_{\emptyset}(\negy)$}};
		\draw (4.9,1) node [right] {{\scriptsize $\Omega_{\{1\}}(\negy)$}};
		\draw (.9,3) node [right] {{\scriptsize $\Omega_{\{2\}}(\negy)$}};
		\draw (4.9,3) node [right] {{\scriptsize $\Omega_{\{1,2\}}(\negy)$}};
		\draw [mark=*, color=black] plot (8,4) node [right] {$\negy-\mathbf{1}$};
		\draw[dashed] (8,0) |- (8,4);
		\draw[dashed] (4,0) |- (4,4);
		\draw[dashed] (0,4) |- (8,4);
		\draw[dashed] (0,2) |- (8,2);
		\end{tikzpicture}
		\caption{The decomposition of a region $\operatorname{C}(\negy-\mathbf{1})$ in $\mathbb{N}_0^2$ into the $4$ disjoint subsets $\Omega_{\emptyset}(\negy)$, $\Omega_{\{1\}}(\negy)$, $\Omega_{\{2\}}(\negy)$, and $\Omega_{\{1,2\}}(\negy)$.}
	\end{figure}
	\label{exaf}

	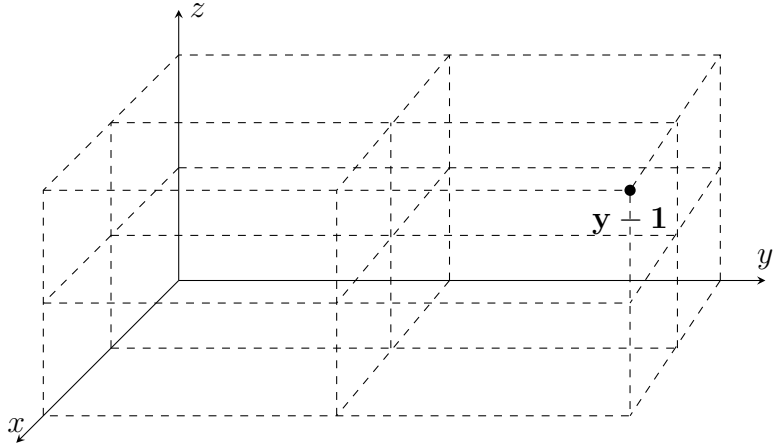
\begin{figure}
		\centering
		\begin{tikzpicture}[x=1cm,y=0.5cm,z=0.3cm,>=stealth, scale=.6]
		\draw[->] (xyz cs:x=0) -- (xyz cs:x=13) node[above] {$y$};
		\draw[->] (xyz cs:y=0) -- (xyz cs:y=12) node[right] {$z$};
		\draw[->] (xyz cs:z=0) -- (xyz cs:z=-12) node[above] {$x$};
		
		\draw[dashed] 
		(xyz cs:z=-10) -- 
		+(0,10) coordinate (u) -- 
		(xyz cs:y=10) -- 
		+(12,0) -- 
		++(xyz cs:x=13,z=-10) coordinate (v) --
		+(0,-10) coordinate (w) --
		cycle;
		\draw[dashed] (u) -- (v);
		\draw[dashed] (12,0) |- (12,10);
		\draw[dashed] (12,0) -- (w);
		\draw[dashed] (xyz cs:z=-5) -- +(0,10);
		\draw[dashed] (xyz cs:x=6.5,z=-10) -- +(0,10);
		\draw[dashed] (xyz cs:x=12.55, z=-5) -- +(0,10);
		\draw[dashed] (xyz cs:x=6.2, z=-5) -- +(0,10);
		\draw[dashed] (xyz cs:x=0,z=-5) -- +(12.5,0);
		\draw[dashed] (xyz cs:x=0,z=-5, y=10) -- +(12.5,0);

		\draw[dashed] (6,0) |- (6,10);
		
		\draw[dashed] (xyz cs:x=6,z=0) -- (xyz cs:x=6.5,z=-10);
		
		\draw[dashed] (xyz cs:x=6,z=0, y=5) -- (xyz cs:x=6.5,z=-10, y=5);
		
		\draw[dashed] (xyz cs:x=6,z=0, y=10) -- (xyz cs:x=6.5,z=-10, y=10);
		
		\draw[dashed] (xyz cs:x=0,z=0, y=5) -- (xyz cs:x=0,z=-10, y=5);
		
		\draw[dashed] (xyz cs:x=12,z=0, y=5) -- (xyz cs:x=13,z=-10, y=5);
		
		\draw[dashed] (xyz cs:x=0,z=0, y=5) -- (xyz cs:x=12,z=0, y=5);
		
		\draw[dashed] (xyz cs:x=0,z=-5, y=5) -- (xyz cs:x=12.55,z=-5, y=5);
		
		\draw[dashed] (xyz cs:x=0,z=-10, y=5) -- (xyz cs:x=13,z=-10, y=5);
		
		\node[fill,circle,inner sep=1.5pt, label={below:$\negy-\mathbf{1}$}] at (v) {};
		\end{tikzpicture}
		\caption{The region $\operatorname{C}(\negy-\mathbf{1})\subset\mathbb{N}_0^3$ splits into the $8$ disjoint sets $\Omega_{\emptyset}(\negy)$, $\Omega_{\{1\}}(\negy)$, $\Omega_{\{2\}}(\negy)$, $\Omega_{\{3\}}(\negy)$, $\Omega_{\{1,2\}}(\negy)$, $\Omega_{\{1,3\}}(\negy)$, $\Omega_{\{2,3\}}(\negy)$, and $\Omega_{\{1,2,3\}}(\negy)$.}
	\end{figure}
\end{remark}

\begin{proposition}\label{cornerunionordinary}
	Let $d \geq 2$, $\negc = (c_1,\ldots,c_d) \in \mathbb{N}^d$,  with $c_j > 1$ for all $j$ and let $J \in \mathcal{P}_d$ be a nonempty set. Then, for any subset $A \subseteq \Omega_{J}(\negc)$,
	$$
	S = A \cup \mathcal{O}(\negc)
	$$
	is a GNS in $\mathbb{N}_0^d$ with corner $\negc$.
\end{proposition}

\begin{proof}
Let $\emptyset \neq J \in \mathcal{P}_d$ and $ A \subseteq \Omega_{J }(\negc)$. Note that $S$ is a GNS. In fact, $\mathbf{0} \in S$ and $\mathbb{N}_0^d \setminus S$ is finite, since $\mathcal{O}(\negc)$ is an ordinary GNS and $\mathcal{O}(\negc) \subseteq S$. Now, let $\negalpha, \negbeta \in S$. If $\negalpha$ or $\negbeta$ lies in $\mathcal{O}(\negc)$, then it is clear that $\negalpha + \negbeta \in S$. If $\negalpha , \negbeta \in A$, then $\negalpha + \negbeta = (\alpha_1 + \beta_1 , \ldots , \alpha_d + \beta_d) \in \mathcal{O}(\negc) \subseteq S$, since $\alpha_{j} + \beta_{j} \geq c_{j}$ for all $j \in J$. Hence, $A \cup \mathcal{O}(\negc)$ is a GNS. Let us prove that $S$ has corner $\negc$. If $i \in [d]$ and $\alpha \geq c_i$, then \linebreak $(\beta_1,\ldots,\beta_{i-1}, \alpha, \beta_{i+1},\ldots,\beta_d) \in \mathcal{O}(\negc) \subseteq S$, since $\negc = (c_1,\ldots,c_d)$ is the corner of $\mathcal{O}(\negc)$. Thus, the condition (1) of Definition \ref{corner} is verified. So, it remains to verify the part (2) of Definition \ref{corner}. Let  $\negalpha = (\alpha_1,\ldots , \alpha_d) \in A$. For $J=[d]$, we have $1 \leq \alpha_i \leq c_i - 1$ for all $i \in [d]$, which implies that $(c_i - 1)\nege_i \notin S$, for all $i \in [d]$. If $J\neq [d]$, then there is $\ell \in [d] \setminus J$ such that $ \alpha_\ell < c_\ell - 1$, and thus $\negc - \mathbf{1} \notin S$. Therefore, we conclude that $\negc$ is the corner of $S$.
\end{proof}

As a consequence, we obtain a lower bound for the number of GNSs in $\mathbb{N}_0^d$ with corner $\negc$ as follows.

\begin{theorem} \label{teo cota GNS}
	Let $d \geq 2$ and $\negc = (c_1,\ldots,c_d)\in \mathbb{N}^d$, with $c_i > 1$ for all $i$, and let $N(\negc)$ be the number of GNSs in $\mathbb{N}_0^d$ with corner $\negc$. Then 
	$$
	\displaystyle  N(\negc) \geq 1 + \sum_{J \in \mathcal{P}_d\setminus \{\emptyset\}} \big( 2^{n_{J}} - 1 \big),
	$$
	where $$n_J=n_J(\negc)=\prod_{j \in J}\textstyle\lfloor \frac{c_j}{2} \rfloor \displaystyle \prod_{t\in [d]\setminus J}\textstyle\lceil \frac{c_t}{2} \rceil.$$
\end{theorem}
\begin{proof}
	For a fixed $J \in \mathcal{P}_d$, we have that 
	$$
	\displaystyle | \{ A \mbox{ : } \emptyset \neq A \subseteq \Omega_{J }(\negc)\}| = \big[  2^{\prod_{j \in J}(c_j - \lceil \frac{c_j}{2} \rceil) \prod_{t \notin J} \lceil \frac{c_t}{2} \rceil} \big] - 1.
	$$
	Note that, by Remark \ref{remark Omega_i}, if $A \neq \emptyset$, then the GNS of the form $A \cup \mathcal{O}(\negc)$ are all distinct. So, as $\mathcal{O}(\negc) \in \mathcal{C}(\negc)$, we can conclude that 
	$$
	\displaystyle | \mathcal{C}(\negc)| \geq 1 + \sum_{J \in \mathcal{P}_d} \big[  2^{\prod_{j \in J}(c_j - \lceil \frac{c_j}{2} \rceil) \prod_{t \notin J} \lceil \frac{c_t}{2} \rceil} - 1 \big].
	$$
	Since $a-\lceil a/2\rceil=\lfloor a/2\rfloor$ for any integer $a$, we obtain the stated formula.
\end{proof}

\begin{remark}
	Despite the idea behind the lower bound given in Theorem~\ref{teo cota GNS} is the same employed by Backelin \cite{backelin}, it is worth to notice why Theorem~\ref{teo cota GNS} cannot be applied to the case $d=1$. The reason is that for $d=1$ the set $\Omega_{\{1\}}$ becomes $\{n\in \mathbb{N} \ : \ \lceil\frac{c}{2}\rceil\leq  n \leq c-1\}$, and hence a subset $A\subseteq \Omega_{\{1\}}$ containing $c-1$ does not give a numerical semigroup $A\cup \mathcal{O}(c)$ with conductor $c$. On the other hand, for $d\geq 2$, the decomposition of $\operatorname{C}(\negc-\mathbf{1})$ into the disjoint regions $\Omega_{J }(\negc)$ takes into account that for each $i\in [d]$ there are gaps of the GNSs $A \cup \mathcal{O}(\negc)$ with at least one coordinate equal to $c_i-1$, for $A\subseteq \Omega_{J}(\negc)$ as in Proposition~\ref{cornerunionordinary}, no matter the choice of $J\in \mathcal{P}_d\setminus \{\emptyset\}$.
\end{remark}

\subsection{An upper bound on the number of GNSs with fixed corner}

	As in the previous subsection, let us consider for each $\negc\in \mathbb{N}^d\setminus \operatorname{C}(\mathbf{1})$ the decomposition of $\operatorname{C}(\negc-\mathbf{1})$ into $2^d$ subsets
\begin{center}
	$
	\Omega_{J}(\negc):= \big\{ \negx \in \mathbb{N}_0^d \ : \ \big\lceil \frac{c_{j}}{2} \big\rceil  \leq x_{j} \leq  c_{j} - 1 \mbox{ for } j \in J \mbox{ and } x_i < \big\lceil \frac{c_i}{2} \big\rceil \mbox{ for } i \in [d] \setminus J  \big\}
	$,
\end{center}
where $J\in \mathcal{P}_d$, the power set of $[d]$. We note that, for every $\negx\in \Omega_\emptyset(\negc)\setminus \{\mathbf{0}\}$, there exists $n\in \mathbb{N}$ such that 
$$n\negx \in \operatorname{C}(\negc-\mathbf{1}) \setminus \Omega_\emptyset(\negc).$$
The following result is about configurations of points in $\operatorname{C}(\negc-\mathbf{1})$ that do not provide GNSs.

\begin{lemma}
	Let $\negc = (c_1,\ldots,c_d) \in \mathbb{N}^d$,  with $c_j > 1$ for all $j \in [d]$. Then, for any nonempty subset $A \subseteq \Omega_{\emptyset}(\negc)$ with $A\neq \{\mathbf{0}\}$, there are $|\negc|-|\Omega_\emptyset(\negc)|-1$ subsets $B$ of $\operatorname{C}(\negc-\mathbf{1})\setminus \Omega_\emptyset(\negc)$ such that
	$$
	A \cup B \cup \mathcal{O}(\negc)
	$$
	is not a GNS in $\mathbb{N}_0^d$.
\end{lemma}
\begin{proof}
	If $\negx$ is an element in a nonempty $A \subseteq \Omega_{\emptyset}(\negc)$, then $n\negx \in \operatorname{C}(\negc-\mathbf{1}) \setminus \Omega_\emptyset(\negc)$ for some $n\in \mathbb{N}$. Hence, for any subset $B$ of $\operatorname{C}(\negc-\mathbf{1}) \setminus \Omega_\emptyset$ that does not contain $n\negx$, we have that the $A \cup B \cup \mathcal{O}(\negc)$ is not a GNS since it is not closed to addition. As $|\operatorname{C}(\negc-\mathbf{1}) \setminus \Omega_\emptyset(\negc)|=|\negc|-|\Omega_\emptyset(\negc)|$, there are exactly $|\negc|-|\Omega_\emptyset(\negc)|-1$ such subsets of $\operatorname{C}(\negc-\mathbf{1})\setminus \Omega_\emptyset(\negc)$.
\end{proof}

A consequence of the previous lemma is the following upper bound for $N(\negc)$.

\begin{theorem} \label{teo cota sup GNS}
	For $\negc = (c_1,\ldots,c_d)\in \mathbb{N}^d$, with  $c_i > 1$ for all $i=1,\ldots,d$, let $N(\negc)$ be the number of GNSs in $\mathbb{N}_0^d$ with corner $\negc$. Then 
	$$
	\displaystyle  N(\negc) \leq  2^{|\negc|-1}-(2^{|\Omega_\emptyset(\negc)|-1}-1)\cdot 2^{|\negc|-|\Omega_\emptyset(\negc)|-1}.
	$$
\end{theorem}
\begin{proof}
	In order to bound $N(\negc)$ we will count the configurations of points in $\operatorname{C}(\negc-\mathbf{1})$ which do not provide GNSs. To begin with, we note that there are $2^{|\Omega_\emptyset(\negc)|-1}-1$ possibilities of nonempty subsets $A$ in $\Omega_{\emptyset}(\negc)$, different of $\{\mathbf{0}\}$. From each one of these subsets $A$, there are $2^{|\negc|-|\Omega_\emptyset(\negc)|-1}$ possibilities of subsets $B$ in $\operatorname{C}(\negc-\mathbf{1})\setminus \Omega_\emptyset(\negc)$ such that $A\cup B\cup \mathcal{O}(\negc)$ is not a GNS, by the previous lemma. Hence, putting together these possibilities, there are at least $(2^{|\Omega_\emptyset(\negc)|-1}-1)\cdot 2^{|\negc|-|\Omega_\emptyset(\negc)|-1}$ configurations of points in $\operatorname{C}(\negc-\mathbf{1})\setminus \{\mathbf{0}\}$ that when joined to $\mathcal{O}(\negc)$ do not provide GNSs. Now, since $\operatorname{C}(\negc-\mathbf{1})\setminus\{\mathbf{0}\}$ has $|\negc|- 1$ elements, there are $2^{|\negc|-1}$ possibilities of sets containing $\mathbf{0}$ in $\operatorname{C}(\negc-\mathbf{1})$. From this amount, by using the lower bound on subsets of $\operatorname{C}(\negc-\mathbf{1})$ that do not give GNSs, we obtain that the number of GNSs with corner $\negc$ is upper bounded by $2^{|\negc|-1}-(2^{|\Omega_\emptyset(\negc
		)|-1}-1)\cdot 2^{|\negc|-|\Omega_\emptyset(\negc)|-1}$.
\end{proof}

Next, we present Table \ref{tabela} with the lower and upper bounds obtained in Theorems \ref{teo cota GNS} and \ref{teo cota sup GNS} and the exact values of $N(\negc)$, which has been computed using the Algorithm 1, implemented in GAP \cite{GAP} with the package \texttt{numericalsgps} \cite{numericalsgps}. Observe that every permutation in the coordinates of a given $\negc$ provides the same lower bound, upper bound and exact value of $N(\negc)$.

\begin{table}[]
\caption{Lower bound (LB), upper bound (UB) and exact values for $N(\negc)$}
\label{tabela}
	\begin{tabular}{||c|c|c|c||c|c|c|c||}
		\hline
		$\negc$ & LB & $N(\negc)$ & UB & $\negc$ & LB & $N(\negc)$ & UB \\
		\hline
		$(2,2)$ & 4  & 4 &	 8	& $(6,3)$ & 78 & 3,212 & 67584 \\
		$(3,2)$ & 6  & 10 & 24 & $(5,4)$ & 94 & 8,758 & 270336 \\ 
		$(3,3)$ & 8  & 38 & 144 & $(2,2,2)$ & 8 & 52 &	128 \\
		$(4,2)$ & 10 & 30 & 96 & $(3,2,2)$ & 14 & 388 & 1536 \\
		$(4,3)$ & 22 & 203 & 1,152 & $(4,2,2)$ & 22 & 2,903 & 24576 \\
		$(5,2)$ &  14  & 66 &	 320 & $(3,3,2)$ & 30 & 6,930 & 73,728 \\
		$(6,2)$ & 22 & 199 & 1,280 & $(4,3,2)$ & 58 & 136,277 & 4,718,592 \\
		$(5,3)$ & 26 & 669 & 8,448 & $(2,2,2,2)$ & 16 & 4,382 & 32,768 \\
		$(4,4)$ & 46 & 1,587 & 18,432 & $(3,2,2,2)$ & 30 & 222,734 &	 6,291,456 \\
		\hline
	\end{tabular}
\end{table}

\section{Concluding remarks}

In the preceding sections, we addressed the natural relations of the corner element with other invariants of a GNS, the problem of computing all GNSs with fixed corner, and  basic estimates on the number of such GNSs. As naturally occurs when an invariant is introduced, many questions arise. We list here some of them. What is the magnitude of the number of GNSs in $\mathbb{N}_0^d$ with a fixed corner? We saw that these GNSs can be divided into two classes: what can be said about the proportion of Frobenius and non-Frobenius GNSs with fixed corner? Apart from some known families of Frobenius GNS, which other families of GNSs could be described in terms of the corner element? In the spirit of the recent advances in numerical semigroups (see the surveys \cite{Delgado, K}), is there any approach involving the corner element for counting GNSs in $\mathbb{N}_0^d$ by genus or dealing with the generalized Wilf conjecture?

\end{document}